\documentclass[a4paper,10pt]{amsart}

\usepackage[T1]{fontenc} \usepackage[utf8]{inputenc}
\usepackage{lmodern}

\usepackage{enumerate}

\usepackage{epsfig} \input{xy} \xyoption{all}

\usepackage{amsfonts,amsmath,amstext,amsbsy,amssymb}
\usepackage{amsbsy,amsopn,amsthm,amscd,amsxtra}

\usepackage{latexsym,mathrsfs}

\usepackage{mathabx}

\usepackage{hyperref}
\RequirePackage[dvipsnames]{xcolor} % [dvipsnames]
\definecolor{halfgray}{gray}{0.55} 
\definecolor{webgreen}{rgb}{0,0.5,0}
\definecolor{webbrown}{rgb}{.6,0,0} \hypersetup{%
  colorlinks=true, linktocpage=true, pdfstartpage=3,
  pdfstartview=FitV,%
  breaklinks=true, pdfpagemode=UseNone, pageanchor=true,
  pdfpagemode=UseOutlines,%
  plainpages=false, bookmarksnumbered, bookmarksopen=true,
  bookmarksopenlevel=1,%
  hypertexnames=true,
  pdfhighlight=/O,%hyperfootnotes=true,%nesting=true,%frenchlinks,%
  urlcolor=webbrown, linkcolor=RoyalBlue,
  citecolor=webgreen, %pagecolor=RoyalBlue,%
  pdftitle={},%
  pdfauthor={},%
  pdfsubject={2000 MAthematical Subject Classification: Primary:},%
  pdfkeywords={},%
  pdfcreator={pdfLaTeX},%
  pdfproducer={LaTeX with hyperref}%
}

\newcommand{\abs}[1]{\left\lvert{#1}\right\rvert}
\newcommand{\norm}[1]{\left\|{#1}\right\|}

\newcommand{\bb}{\mathbb} 
\newcommand{\mc}{\mathcal} 

\newcommand{\R}{\mathbb{R}}\newcommand{\N}{\mathbb{N}}
\newcommand{\Z}{\mathbb{Z}}

 \newcommand{\ie}{i.e.\ }
\newcommand{\eg}{e.g.\ }

\newtheorem*{thm*}{Theorem} 
\newtheorem*{thmA}{Theorem A} \newtheorem*{thmB}{Theorem B}

\newtheorem{theorem}{Theorem}[section]
\newtheorem{proposition}[theorem]{Proposition}
\newtheorem{lemma}[theorem]{Lemma}
\newtheorem{corollary}[theorem]{Corollary}

\newtheorem{question}[theorem]{Question}

\theoremstyle{definition} \newtheorem{definition}[theorem]{Definition}

\theoremstyle{remark} \newtheorem{remark}[theorem]{Remark}
\newtheorem{claim}{Claim}

\newcommand{\Diff}[2]{\mathrm{Diff}_{#1}^{#2}}
\newcommand{\Homeo}[1]{\mathrm{Homeo}_{#1}}

\newcommand{\eps}{\varepsilon} \newcommand{\carr}{\righttoleftarrow}
\newcommand{\Hol}{\mathcal{H}}

 \newcommand{\dd}{\:\mathrm{d}}

\newcommand{\dpf}{\partial_{\mathrm{fib}}}

\DeclareMathOperator{\M}{\mathfrak{M}} \DeclareMathOperator{\Lip}{Lip}
\DeclareMathOperator{\Per}{Per} \DeclareMathOperator{\Fix}{Fix}
 \DeclareMathOperator{\GL}{GL}

\title[Liv\v{s}ic theorem]{Liv\v{s}ic theorem for low-dimensional
  diffeomorphism cocycles}

\author[A. Kocsard]{Alejandro Kocsard} \address{Instituto de
  Matem\'{a}tica e Estat\'{\i}stica - Universidade Federal
  Fluminense. Rua M\'{a}rio Santos Braga S/N, 24.020-140 Niter\'{o}i,
  RJ, Brasil}\urladdr{www.professores.uff.br/kocsard}
\email{akocsard@id.uff.br}

\author[R. Potrie]{Rafael Potrie} \address{CMAT, Facultad de Ciencias,
  Universidad de la Rep\'{u}blica, Uruguay}
\urladdr{www.cmat.edu.uy/$\sim$rpotrie} \email{rpotrie@cmat.edu.uy}

\date{\today}

\begin{document}

\maketitle

\begin{abstract}
  We prove a Liv\v{s}ic type theorem for cocycles taking values in
  groups of diffeomorphisms of low-dimensional manifolds. The results
  hold without any localization assumption and in very low
  regularity. We also obtain a general result (in any dimension) which
  gives necessary and sufficient conditions to be a coboundary.
\end{abstract}

\section{Introduction}\label{SectionIntroduccion}

In the study of hyperbolic dynamical systems there is a general (and
vague) idea that can be summarized with the following sentence:
\begin{quote}
  \emph{Most of the dynamical interesting information on a hyperbolic
    system is concentrated in its periodic orbits.}
\end{quote}

An archetypal example of a result supporting this idea is the
celebrated Liv\v{s}ic's theorem
~\cite{LivCertPropHomol,LivsicCohomDynSys} claiming that given a
hyperbolic homeomorphism $f\colon M\carr$, a H\"{o}lder function
$\Phi\colon M\to\R$ is a \emph{coboundary}, \ie there exists a
continuous function $u\colon M\to\R$ satisfying
\begin{displaymath}
  u\circ f-u=\Phi,
\end{displaymath}
if and only if
\begin{displaymath}
  \sum_{j=0}^{n-1}\Phi(f^j(p))=0,
\end{displaymath}
for every periodic point $p\in M$, with $f^n(p)=p$. 

Due to the interest this result has received since its appearance and
the large amount of consequences that follow from it, several
generalizations have been studied. Some of them consider more general
dynamics on the base. For instance, in the works
\cite{KatokKononenkoCocStab,WilkinsonCohomoEq} the cohomology of real
cocycles over partially hyperbolic systems is analyzed.

In this paper, we consider a different kind of generalization: given a
complete metric group $G$, a \emph{$G$-cocycle} is just a continuous
map $\Phi\colon M\to G$ and one wants to determine whether the
condition
\begin{displaymath}
  \Phi\big(f^{n-1}(p)\big)\Phi\big(f^{n-2}(p)\big)\ldots\Phi(p) = e_G,
  \quad\forall p\in\Fix(f^n),
\end{displaymath}
where $e_G$ is the identity element of $G$, is not just necessary but
also sufficient to guaranty the existence of a ``transfer function''
$u\colon M\to G$ satisfying
\begin{displaymath}
  \Phi(x)=u(f(x))u(x)^{-1},\quad\forall x\in M.
\end{displaymath}

Liv\v{s}ic himself gave in \cite{LivCertPropHomol} an affirmative
answer to this question for cocycles taking values on a topological
group admitting a complete bi-invariant distance (\eg Abelian or
compact groups). However, the general situation is considerably much
more complicated.

So far, the main technique to handle this problem when the group
$G$ does not admit a bi-invariant metric has consisted in considering
a left-invariant metric on $G$ and to try to control the distortion
produced by right translations to be able to apply the very same
scheme of proof used in the Abelian case.

In order to get such a control of the distortion of the distance, some
\emph{localization hypotheses} have been considered in the
literature. For instance, in \cite{LivsicCohomDynSys} Liv\v{s}ic gave
a positive answer to above question for linear cocycles (\ie where
$G=\GL_d(\R)$) which are not too far away from the identity constant
cocycle. Improvements of these results using weaker localization
hypotheses have been obtained for cocycles taking values in arbitrary
finite-dimensional Lie groups (see
\cite{delaLlaveWindsorLivThmNonComm, NiticaKatokRigHighRankI} and
references therein) until the recent complete solution of the global
Liv\v{s}ic problem for linear cocycles \cite{KalininLivThmMat} (see
also the recent preprint \cite{GrabarnikGuysinsky}).

In the infinite dimensional case, particularly when $G$ is a group of
diffeomorphisms of a compact manifold, the study began with the
seminal paper of Ni\c{t}ic\u{a} and T\"or\"ok
\cite{NiticaTorokCohomolDynSyst}. The diffeomorphism groups seem to be
the most interesting infinite dimensional groups for applications to
rigidity theory (see \cite{NiticaKatokRigHighRankI} and references
therein).

It is important to remark that, in contrast with the finite
dimensional case, all the results for groups of diffeomorphisms
obtained so far (see \cite{delaLlaveWindsorLivThmNonComm} for a survey
with references) involve non-sharp localization hypotheses (in the
sense that not every coboundary satisfies them) and require higher
regularity for the diffeomorphism group (\ie $G=\Diff{}r(N)$, with
$r\geq 4$). Moreover, the control of distortion techniques used in
\cite{NiticaTorokCohomolDynSyst, delaLlaveWindsorLivThmNonComm} yield
a loss of regularity in the solution of the cohomological equation. A
recent result for infinite dimensional groups which does not fit in
the previous description is due to Navas and
Ponce~\cite{NavasPonceLivAnalGerms}. They prove a \emph{Liv\v{s}ic
  theorem} for cocycles taking values in the group of analytic germs
at the origin.

In this paper we use completely different techniques, with a more
geometric flavor, which allow us to deal with the low regularity case
(cocycles can take values in the group of $C^1$-diffeomorphisms). The
main novelty of our approach is Theorem~\ref{thm:lyap-exp-dom-cobound}
which can be of independent interest. Regarding this result within the
context of localization arguments, we could say that it is proven that
our \emph{non-uniform localization hypothesis} (\ie vanishing of
fibered Lyapunov exponents) is equivalent to be a coboundary.

In \S\ref{sec:POO-implies-zero-Lyap} we show that in the
low-dimensional case the periodic orbit condition implies the
vanishing of fibered Lyapunov exponents, proving in such cases the
general (or global) \emph{Liv\v{s}ic theorem} for groups of
diffeomorphisms. We conjecture that such a result holds in any
dimension.

\subsection{Main results}
\label{sec:main-res}

The main results of this article are the following Liv\v{s}ic type
theorems which, to the best of our knowledge, are the first general
(\ie global) results for cocycles taking values in groups of
diffeomorphisms of compact manifolds:

\begin{thmA}
  \label{thm:main-thm-T}
  Let $f\colon M\carr$ be a hyperbolic homeomorphism and $\Phi\colon
  M\to\Diff{}1(\R/\Z)$ be an $\alpha$-H\"older cocycle satisfying the
  so called \emph{periodic orbit obstructions}
  \begin{displaymath}
    \Phi(f^{n-1}(p))\circ\Phi(f^{n-2}(p))\circ\cdots\circ\Phi(p)
    =id_{\R/\Z},\quad\forall p\in\Fix(f^n),\ \forall n\in\N.
  \end{displaymath}

  Then, there exists an $\alpha$-H\"older map $u\colon
  M\to\Diff{}1(\R/\Z)$ satisfying
  \begin{displaymath}
    \Phi(x)=u(f(x))\circ u(x)^{-1},\quad\forall x\in M.
  \end{displaymath}
\end{thmA}

\begin{remark}
  The very same argument we use to prove Theorem A works for cocycles
  taking values in $\Diff{}1([0,1])$
\end{remark}

For higher regularity, invoking the main result of
\cite{delaLlaveWindsorSmoorhDependence}, one easily gets the following
consequence of Theorem A:

\begin{corollary}
  \label{cor:main-thm-A}
  Let $f\colon M\carr$ be a hyperbolic homeomorphism and $\Phi\colon
  M\to\Diff{}r(\R/\Z)$, with $r\geq 1$, be an $\alpha$-H\"older
  cocycle satisfying the \emph{periodic orbit obstructions}.

  Then, there exists an $\alpha$-H\"older map $u\colon
  M\to\Diff{}r(\R/\Z)$ satisfying
  \begin{displaymath}
    \Phi(x)=u(f(x))\circ u(x)^{-1},\quad\forall x\in M.
  \end{displaymath}
\end{corollary}

Assuming higher regularity on the dynamics of the base and the
cocycle, applying the results of \cite{JourneRegLem,
  NiticaTorokRegularResultsSolLiv} one can improve the regularity of
the solution of the cohomological equation. Since this kind of results
is beyond the scope of this article, we suggest the interested reader
to consult \cite{NiticaKatokRigHighRankI,WilkinsonCohomoEq} for
further information.

In dimension $2$, we can obtain a similar result for the group of
area-preserving diffeomorphisms:
\begin{thmB}
  \label{thm:main-thm-S} 
  Let $M$ be a smooth closed manifold and $f\colon M\carr$ be a
  $C^{1+\theta}$ transitive Anosov diffeomorphism. Let $S$ denote a
  compact surface, $\omega$ be a Lebesgue probability measure on $S$
  and let $\Diff\omega{r}(S)$ be the group of $C^r$-diffeomorphisms of
  $S$ that leave $\omega$ invariant.

  Let $\Phi\colon M\to\Diff{\omega}{1+\alpha}(S)$ be a
  $C^{1+\alpha}$-cocycle\footnote{This means that the induced map
    $M\times S\ni (x,y)\mapsto \Phi(x)(y)\in S$ is $C^{1+\alpha}$.}
  satisfying the periodic orbit obstructions.

  Then, there exists an $\alpha$-H\"older map $u\colon
  M\to\Diff\omega{1}(S)$ such that
  \begin{displaymath}
    \Phi(x)=u(f(x))\circ u(x)^{-1},\quad\forall x\in M.
  \end{displaymath}
\end{thmB}

The proofs of Theorems A and B consist in two steps. The first one
concerns the vanishing of Lyapunov exponents for the cocycles and
relies heavily in the low-dimensionality of the fibers. The second one
holds in any dimension and can be of independent interest (see
Theorem~\ref{thm:lyap-exp-dom-cobound}).

\subsection*{Acknowledgments}
\label{sec:aknowledgements}

We are grateful to A. Navas and M. Ponce for useful discussions and
bringing these problems to our attention. This paper was finished
during a visit of A.K. to Universidad de la República (Uruguay). 

A.K. is very thankful to R. Markarian for his hospitality. He was
partially supported by CNPq and FAPERJ (Brasil), and Fondo Clemente
Estable (Uruguay).

R.P. was partially supported by CSIC group 618, FCE-3-2011-1-6749 and
Balzan's research project of J.Palis.

\section{Preliminaries and notations}
\label{sec:preel-notat}

\subsection{Hölder continuity}
\label{sec:holder-continuity}

All along this paper, $(M,d)$ will denote a compact metric space. If
$(M^\prime,d^\prime)$ denotes another arbitrary metric space and
$0<\alpha\leq 1$, a map $\psi\colon M\to M^\prime$ is said to be
\emph{$\alpha$-Hölder} whenever
\begin{displaymath}
  \abs{\psi}_\alpha:=\sup_{x\neq y}
  \frac{d^\prime\big(\psi(x),\psi(y)\big)}{d(x,y)^\alpha} <\infty. 
\end{displaymath}
Most of the functions and maps we shall deal with in this paper will
be at least H\"older because, as it was already observed in
\cite{KocCohomC0Inst}, in general $C^0$ regularity is not appropriate
for dynamical cohomology.

When $\alpha<1$, the space of $\alpha$-Hölder maps from $M$ to
$M^\prime$ will be denoted by $C^\alpha(M,M^\prime)$. As usual, we use
the term \emph{Lipschitz} as a synonym of $1$-Hölder, and to avoid
confusions with the differentiable case, we shall write
$C^{\mathrm{\Lip}}(M,M^\prime)$ for the space of Lipschitz functions.

For such a real constant $\alpha$, we can define a new distance on $M$
by
\begin{equation}
  \label{eq:alpha-distance}
  d_\alpha(x,y):=d(x,y)^\alpha,\quad \forall x,y\in M.
\end{equation}
Observe that the topologies induced by $d$ and $d_\alpha$ coincide and
a map $\psi\colon (M,d)\to (M^\prime,d^\prime)$ is $\alpha$-Hölder if
and only if $\psi\colon (M,d_\alpha)\to (M^\prime,d^\prime)$ is
Lipschitz.

\subsection{Borel probability measures}
\label{sec:borel-measures}

Given an arbitrary locally compact metric space $X$, we write $\M(X)$
for the space of Borel probability measures on $X$ and we will always
consider it endowed with (the restriction of) the weak-$\star$
topology. If $Y$ denotes another compact metric space, any continuous
map $h\colon X\to Y$ naturally induces a linear map
$h_\star\colon\M(X)\to\M(Y)$ defined by the following property:
\begin{displaymath}
  \int_Y \phi\dd (h_\star\mu):=\int_X\phi\circ h\dd\mu,
  \quad\forall\phi\in C^0_c(Y),\ \forall\mu\in\M(X).
\end{displaymath}

In this way, if $f\colon X\carr$ is a continuous map, one defines the
space of $f$-invariant measures by
\begin{displaymath}
  \M(f):=\{\mu\in\M(M) : f_\star\mu=\mu\}.
\end{displaymath}

\subsection{Hyperbolic homeomorphisms}
\label{sec:hyper-homeos}

Let $(M,d)$ be a compact metric space and $f\colon M\carr $ be
a homeomorphism. Given any $x\in M$ and $\epsilon >0$, define the
\emph{local stable} and \emph{unstable sets} by
\begin{align*}
  W^s_\eps(x,f) &:= \left\{y\in M : d(f^n(x),f^n(y))\leq\eps,\ \forall
    n \geq 0\right\}, \\
  W^u_\eps(x,f) &:= \left\{y\in M : d(f^n(x),f^n(y))\leq\eps,\ \forall
    n \leq 0\right\},
\end{align*}
respectively. Where there is no risk of ambiguity, we just write
$W^s_\epsilon(x)$ instead of $W^s_\epsilon(x,f)$, and the same holds
for the local unstable set.

Following \cite{AvilaVianaExtLyapInvPrin}, we introduce the following

\begin{definition}
  \label{def:hyperbolic-homeo}
  A homeomorphism $f\colon M\carr$ is said to be \emph{hyperbolic with
    local product structure} whenever there exist constants
  $\eps_0,\delta_0,K_0,\lambda >0$ and functions $\nu_s, \nu_u\colon M
  \to (0,\infty)$ such that the following conditions are satisfied:
  \begin{itemize}
  \item[(h1)] $d(f(y_1),f(y_2)) \leq \nu_s(x) d(y_1,y_2)$, $\forall
    x\in M$, $\forall y_1,y_2 \in W^s_{\eps_0} (x)$;
  \item[(h2)] $d(f(y_1), f(y_2)) \geq \nu_u(x) d(y_1,y_2)$, $\forall
    x\in M$, $\forall y_1,y_2 \in W^u_{\eps_0} (x)$;
  \item[(h3)] $\nu_s^{(n)}(x) := \nu_s(f^{n-1}(x)) \ldots \nu_s(x)
    <K_0e^{-\lambda n}$, $\forall x\in M$, $\forall n\geq 1$;
  \item[(h4)] $\nu_u^{(n)}(x) := \nu_u(f^{n-1}(x)) \ldots \nu_u(x) >
    K_0e^{\lambda n}$, $\forall x\in M$, $\forall n\geq 1$;
  \item[(h5)] If $d(x,y) \leq \delta_0$, then $W^u_{\eps_0}(x)$ and
    $W^s_{\eps_0}(y)$ intersect in a unique point which is denoted by
    $[x,y]$, and it depends continuously on $x$ and $y$.
  \end{itemize}
\end{definition}

\begin{remark}
  For the sake of simplicity of the exposition and to avoid
  unnecessary repetitions, from now on we shall assume that all
  hyperbolic homeomorphisms are transitive and exhibit local product
  structure.
\end{remark}

For such homeomorphisms, one can define the \emph{stable} and
\emph{unstable sets} by
\begin{displaymath}
  W^s(x,f):= \bigcup_{n\geq 0} f^{-n}\big(W^s_\epsilon(f^n(x))\big)
  \quad\text{and}\quad W^u(x,f):= \bigcup_{n\geq 0} 
  f^{n}\big(W^u_\epsilon(f^{-n}(x))\big),
\end{displaymath}
respectively.

Notice that shifts of finite type and basic pieces of Axiom A
diffeomorphisms are particular examples of hyperbolic homeomorphisms
with local product structure (see for instance \cite[Chapter
IV,\S9]{ManeBook} for details).

\begin{remark}
  \label{rem:change-metric}
  For our purposes, it is important to notice that the notion of
  hyperbolicity for homeomorphisms is invariant under Hölder changes
  of metric. More precisely, a homeomorphism $f\colon (M,d)\carr$ is
  hyperbolic if and only $f\colon (M,d_\alpha)\carr$ is hyperbolic,
  for any $\alpha\in (0,1)$, where the distance $d_\alpha$ is defined
  by \eqref{eq:alpha-distance}.
\end{remark}

The following result is proven for locally maximal hyperbolic sets of
smooth diffeomorphisms in \cite[Chapter 6]{KatokHasselblatt}. However,
by inspection on the proof (see \cite[Corollary 6.4.17 and Proposition
6.4.16 ]{KatokHasselblatt}) it can be easily check that the very same
proof works for hyperbolic homeomorphisms with local product structure
(see also \cite[page 1026]{KalininLivThmMat}):

\begin{theorem}[Anosov closing lemma]
  \label{thm:AnosovClosing} Let $f\colon(M,d)\carr$ be a hyperbolic
  homeomorphism\footnote{Remember that by our assumption, hyperbolic
    homeomorphisms have local product structure.}. Then, there exist
  constants $c,\delta_1 >0$ such that for every $x\in M$ and any $n>0$
  satisfying $d(x,f^n(x))< \delta_1$, there exist unique points $p\in
  \Fix(f^n)$ and $y\in M$ such that:
  \begin{enumerate}
  \item $d(f^i(x),f^i(p))\leq c
    d(x,f^n(x))e^{-\lambda\min{\{i,n-i\}}}$;
  \item $d(f^i(p),f^i(y))\leq c d(x,f^n(x))e^{-\lambda i}$;
  \item $d(f^i(x),f^i(y))\leq c d(x,f^n(x))e^{-\lambda(n-i)}$;
  \end{enumerate}
  for every $i\in\{0, \ldots, n-1\}$, where $\lambda>0$ is the
  constant given in Definition~\ref{def:hyperbolic-homeo}.
\end{theorem}

\begin{remark}
  \label{rem:x-y-p-AnosovClosing}
  Notice that by uniqueness, we have that $y=[x,p]$, where the
  brackets $[\cdot,\cdot]$ are given by
  Definition~\ref{def:hyperbolic-homeo}.
\end{remark}

\subsection{Cocycles and coboundaries}
\label{sec:cocycles-coboudaries}

Let $G$ denote a topological group whose topology is induced by a
complete distance function $d_G$, and let $f\colon (M,d)\carr$ be a
homeomorphism.

In this work all cocycles we consider will be at least continuous. In
fact, a $G$-\emph{cocycle} (over $f$) is just a continuous map
$\Phi\colon M\to G$. As usual, we use the following notation
\begin{displaymath}
  \Phi^{(n)}(x):=
  \begin{cases}
    e_G,&\text{if } n=0;\\
    \Phi(f^{n-1}(x))\Phi^{(n-1)}(x),&\text{if } n>0;\\
    \big(\Phi^{(-n)}(f^n(x))\big)^{-1},&\text{if } n<0;
  \end{cases}
\end{displaymath}
where $e_G\in G$ denotes the identity element of $G$. We say that
$\Phi$ is a H\"older cocycle when $\Phi\colon (M,d)\to(G,d_G)$ is an
$\alpha$-Hölder map, for some $\alpha\in (0,1]$.

A $G$-cocycle $\Phi$ is said to be a $G$-\emph{coboundary} when there
exists a continuous map $u\colon M \to G$ such that
\begin{displaymath}
  \Phi(x) = u(f(x)) \cdot (u(x))^{-1},\quad \forall x \in M.
\end{displaymath}
Notice that in this case it clearly holds
$\Phi^{(n)}(x)=u(f^n(x))\cdot(u(x))^{-1}$, for any $n\in\Z$ and any
$x\in M$. 

The first natural family of obstructions one encounters for a
$G$-cocycle to be a $G$-coboundary is called \emph{periodic orbit
  obstructions} (just \emph{POO} for short):
\begin{equation}
  \label{eq:POO}\tag{POO}
  \Phi^{(n)}(p) = e_G, \quad\forall n\geq 1,\ \forall p \in
  \Fix(f^n).
\end{equation}

In this work we shall mainly concentrate in the case where
$G=\Diff{}{1}(N)$. To deal with this case, we need a slight
generalization of the concept of cocycles that we introduce in the
following paragraph.

\subsection{Fiber bundle maps and cocycles}
\label{sec:fib-map-cocyc}

Let $N$ denote a compact differentiable manifold and $(M,d)$ be
compact metric space as above. Given any $\alpha\in (0,1]$ and $r\geq
0$, a \emph{$C^{\alpha,r}$-fiber bundle} over $M$ with fiber $N$ is an
object $N\to\mc{E}\xrightarrow{\pi} M$, where $\mc{E}$ is a
topological space and $\pi$ is a surjective $\alpha$-H\"older map such
that there exists a finite open cover $\{U_j\}_1^n$ of $M$ with the
following properties:
\begin{itemize}
\item For each $j\in\{1,\ldots,n\}$, there exists a homeomorphism
  $\varphi_j\colon \pi^{-1}(U_j) \to U_j \times N$;
\item If $U_i\cap U_j \neq\emptyset$, there is an $\alpha$-H\"older
  map $g_{ij}\colon U_i\cap U_j\to\Diff{}r(N)$ such that
  \begin{displaymath}
    \varphi_i \circ \varphi^{-1}_j(x,y) = (x, (g_{ij}(x))(y)),
    \quad\forall (x,y)\in\ (U_i \cap U_j) \times N.
  \end{displaymath}
\end{itemize}

As usual, one defines \emph{the fiber over $x$} by
$\mc{E}_x:=\pi^{-1}(x)\subset\mc{E}$ and, due to the fiber bundle
properties, it can naturally endowed with a $C^r$-differentiable
structure turning it into a $C^r$-manifold $C^r$-diffeomorphic to $N$.

As usual, when $N=\R^d$ and the change of coordinate maps $g_{ij}$
given above are $\alpha$-Hölder and have their image contained in
$\GL_d(\R)$, we say that $\R^d\to\mc{E}\xrightarrow{\pi}M$ is a
$C^\alpha$-vector bundle.

The total space of any $C^{\alpha,r}$-fiber bundle
$N\to\mc{E}\xrightarrow{\pi} M$ can be endowed with a distance
function $d_{\mc{E}}$ constructed as follows:

Let $d_N$ be a distance function compatible with the smooth structure
of the fiber manifold $N$, $\{U_j\}_{1}^n$ and $\{\varphi_j\}_1^n$ be
a local trivialization atlas as above, $L>0$ be the Lebesgue number of the
open covering $\{U_j\}_1^n$ and define
\begin{equation}
  \label{eq:d_E-definition}
  d_{\mc{E}}(\zeta,\eta):= \min\bigg\{
  d\big(\pi(\zeta),\pi(\eta)\big) + \inf_{1\leq j\leq n}
  \left\{d_N\Big(\mathrm{pr}_2\big(\varphi_j(\zeta)\big),
    \mathrm{pr}_2\big(\varphi_j(\eta)\big)\Big)\right\}; L \bigg\}. 
\end{equation}
where $\mathrm{pr}_2\colon M\times N\to N$ denotes the projection on
the second coordinate, and by convention, we declare that
$d_N\Big(\mathrm{pr}_2\big(\varphi_j(\zeta)\big),
\mathrm{pr}_2\big(\varphi_j(\eta)\big)\Big)=\mathrm{diam}_{d_N} N$
whenever $\zeta$ or $\eta$ does not belong to $\pi^{-1}(U_j)$.

A \emph{Riemannian structure} on a $C^{\alpha,r}$-fiber bundle
$\mc{E}$ consists on choosing a Riemannian metric on each fiber
$\mc{E}_x$ which varies H\"older-continuously with $x\in M$.

From now on, we will assume every fiber bundle is endowed with a fixed
Riemannian structure and a distance function constructed as above. It
is important to remark that all the concepts we will consider about
fiber bundles are completely independent of these chosen structures.

In this setting, given another $C^{\alpha,r}$-fiber bundle
$N\to\tilde{\mc{E}}\xrightarrow{\tilde\pi} M$ and a
$C^\alpha$-homeomorphism $f\colon M\carr$, a
\emph{$C^{\alpha,r}$-bundle map} over $f$ is a homeomorphism $F\colon
\mc{E}\to\tilde{\mc{E}}$ satisfying $\tilde\pi \circ F = f \circ \pi$
and such that the map $F_x:= F\big|_{\mc{E}_x}\colon
\tilde{\mc{E}}_x\to\mc{E}_{f(x)}$ is a $C^r$-diffeomorphism, for every
$x\in M$.

As usual, the fiber bundle $N\to \mc{E}=M\times
N\xrightarrow{\mathrm{pr}_1} M$, where $\mathrm{pr}_1\colon M\times
N\to M$ denotes the projection on the first coordinate, is said to be
\emph{trivial.}

In this case, any $C^\alpha$ $\Diff{}r(N)$-cocycle $\Phi\colon
M\to\Diff{}r(N)$ naturally induces a $C^{\alpha,r}$-bundle map over
any $f\in\Homeo{}(M)$ defining $F=F_{\Phi,f}\colon M\times N\carr$ by
\begin{equation}
  \label{eq:induced-skew-product}
  F(x,y):=\big(f(x),\Phi(x)(y)\big),\quad\forall (x,y)\in M\times N.
\end{equation}
Observe that in such a case $F_x=\Phi(x)$, for every $x\in M$.

Such a particular bundle map is usually called the \emph{skew-product}
induced by $\Phi$ and $f$. So, bundle maps can be considered as
generalizations of cocycles taking values in groups of
diffeomorphisms.

\subsection{POO and coboundaries for bundle maps}
\label{sec:poo-cobound-fiber-maps}

We can easily extends the notion of \emph{periodic orbit obstructions}
to fiber bundle maps on fiber bundles which, \emph{a priori}, are not
necessarily trivial:

A bundle map $F\colon\mc{E}\carr$ over $f\colon M\carr$ is said to
satisfy the \emph{periodic orbit obstructions} whenever, for every
$n\geq 1$ it holds
\begin{equation}\tag{POO'}
  \label{eq:POO-bund-map}
  F^n_p(\zeta)=\zeta,\quad \forall p\in\Fix(f^n),\ \forall
  \zeta\in\mc{E}_p. 
\end{equation}

Now we finish this paragraph extending the notion of \emph{coboundary}
for (\emph{à priori} more general) fiber bundle maps: if
$N\to\mc{E}\xrightarrow{\pi} M$ denotes a $C^{\alpha,r}$-fiber bundle,
a $C^{\alpha,r}$-bundle map $F\colon\mc{E}\carr$ is said to be a
\emph{$C^{\alpha,s}$-coboundary}, with $s\leq r$, when there exists a
$C^{\alpha,s}$-bundle map $H\colon\mc{E}\to M\times N$ over the
identity map $id\colon M\carr$ such that the following diagram
commutes:

\begin{equation}
  \label{eq:diag-cobound-trivial}
  \xymatrix{& M\times N\ar[rr]^{f\times
      id_N}\ar'[d][dd]^(.35){\mathrm{pr}_1} & & M\times
    N\ar[dd]^{\mathrm{pr}_1} \\
    \mc{E}\ar[rr] ^(.6){F}\ar[dd]^\pi\ar[ru]^H &  &
    \mc{E}\ar[dd]^(.35){\pi}\ar[ru]^{H} &  
    \\
    & M\ar'[r][rr]^(.3){f} & & M \\
    M\ar[ru]^{id}\ar[rr]^f & & M\ar[ru]^{id}  & }
\end{equation}

Observe that with our definition, any fiber bundle admitting a
coboundary bundle map is \emph{\`a posteriori} trivial (see
Theorem~\ref{thm:lyap-exp-dom-cobound} for details and a motivation
for this definition).

\subsection{Lyapunov exponents for bundle maps}
\label{sec:lyap-expon-fiber-bundle-maps}

Let $\pi\colon\mc{E}\to M$ be a $C^{\alpha,r}$-fiber bundle (endowed
with a Riemannian structure) and consider a $C^{\alpha,r}$-bundle map
$F\colon\mc{E}\carr$ over $f\colon M\carr$. Given any point
$\zeta\in\mc{E}$ and $n\in\Z$, we have the linear map
\begin{equation}
  \label{eq:dpf-definition}
  D\left(F^n\big|_{\mc{E}_{\pi(\zeta)}}\right)_\zeta\colon T_\zeta
  \mc{E}_{\pi(\zeta)} \to T_{F^n(\zeta)}\mc{E}_{f^n(\pi(\zeta))}
\end{equation}
between normed vector spaces, and hence it makes sense to talk about
its norm. For the sake of simplicity, such linear operator will be
just denoted by $\dpf F^n(\zeta)$.

Then, one defines the \emph{extremal Lyapunov exponents} of $F$ along
the fibers at $\zeta\in\mc{E}$ by
\begin{displaymath}
  \begin{split}
    \lambda^+(F,\zeta) &= \lim_{n\to+\infty} \frac{1}{n}
    \log\norm{\dpf F^n(\zeta)}; \\
    \lambda^-(F,\zeta) &= \lim_{n\to+\infty} \frac{1}{n} \log
    \norm{\big(\dpf F^n(\zeta)\big)^{-1}}^{-1};
  \end{split}
\end{displaymath}
whenever these limits exist. As a consequence of the sub-additive
ergodic theorem, it is well-known that theses limits exist
almost-everywhere with respect to any $F$-invariant probability
measure.

Given any $\hat\mu\in\M(F)$, one defines the \emph{extremal Lyapunov
  exponents of} $F$ \emph{with respect to} $\hat \mu$ by
\begin{displaymath}
  \lambda^\pm(F, \hat\mu):=\int_{\mc{E}}\lambda^\pm(F,\zeta)
  \dd\hat\mu(\zeta).
\end{displaymath}

The topological version of the invariance principle of Avila and Viana
\cite{AvilaVianaExtLyapInvPrin} gives strong consequences of vanishing
of the extremal Lyapunov exponents of certain invariant measures.

\subsection{Dominated bundle maps}
\label{sec:dominated-fiber-bundle-maps}

As above, let us suppose $F\colon\mc{E}\carr$ is a
$C^{\alpha,r}$-bundle map over a transitive hyperbolic homeomorphism
$f\colon M\carr$. Given $\beta>0$, we say that $F$ is
$(u,\beta)$-\emph{dominated} whenever there exists $\ell\geq 1$ such
that
\begin{equation}
  \label{eq:u-domination-def}
  \norm{\dpf F^{\ell}(\zeta)}
  \leq\frac{\Big(\nu^{(\ell)}_u\big(\pi(\zeta)\big)\Big)^\beta}{2}, 
  \quad\forall \zeta\in\mc{E},
\end{equation}
where $\nu_u$ is the (multiplicative) cocycle over $f$ given in
Definition~\ref{def:hyperbolic-homeo}.  

Analogously, one says that $F$ is said to be
\emph{$(s,\beta)$-dominated} when there exists $\ell\geq 1$ such that
\begin{equation}
  \label{eq:s-domination-def}
  \norm{(\dpf F^{\ell}(\zeta))^{-1}}^{-1}\geq
  2\Big(\nu_s^{(\ell)}\big(\pi(\zeta)\big)\Big)^\beta,
  \quad\forall\zeta\in\mc{E}.
\end{equation}

And $F$ is just said to be \emph{$\beta$-dominated} if it is
simultaneously both $(s,\beta)$- and $(u,\beta)$-dominated. It is
important to remark that the definition of domination implicitly
assumes that the dynamics on the base space is given by a hyperbolic
homeomorphism $f\colon M\carr$.

The following result is a consequence of the classical graph transform
arguments used in \cite{HirshPughShubInvMan} (see \cite[Proposition
5.1]{AvilaVianaExtLyapInvPrin} for an indication of the proof in this
exact context):
\begin{proposition}
  \label{prop:ConosVariedadesInestables}
  If $F$ is an $(s,1)$-dominated $C^{\mathrm{Lip},1}$-fiber bundle
  over a hyperbolic homeomorphism $f\colon M\righttoleftarrow$, then
  there exists a unique partition of $\mc{E}$
  \begin{displaymath}
    \mc{W}^s = \left\{ \mc{W}^s(\zeta)\subset\mc{E} : \zeta\in
      \mc{E}\right\},
  \end{displaymath}
  which satisfies the following conditions:
  \begin{enumerate}
  \item[(i)] for every $\zeta\in\mc{E}$, $\mc{W}^s(\zeta)$ is the
    image of a Lipschitz section $W^s\big(\pi(\zeta),f\big)\to\mc{E}$
    whose Lipschitz constant is uniform, \ie it can be chosen
    independently of $\zeta$;
  \item[(ii)] it is $F$-invariant, \ie
    \begin{displaymath}
      F\big(\mc{W}^s(\zeta)\big)=\mc{W}^s\big(F(\zeta)\big),
      \quad\forall\zeta\in\mc{E}.
    \end{displaymath}
  \end{enumerate}

  Of course, if $F$ is $(u,1)$-dominated, a completely analogous
  result holds and in such a case the ``unstable'' partition is
  denoted by $\mc{W}^u$.
\end{proposition} 

The following result is a combination of Proposition 5.1 and Theorem D
of \cite{AvilaVianaExtLyapInvPrin}:
\begin{theorem}
  \label{thm:inv-principle}
  Let $F\colon\mc{E}\righttoleftarrow$ be a $1$-dominated
  $C^{\mathrm{Lip},1}$-fiber bundle map over a hyperbolic
  homeomorphism $f\colon M\carr$ and let $\hat\mu$ be an $F$-invariant
  probability measure whose projection $\mu:=\pi_\star(\hat\mu)$ to
  $M$ has local product structure and full support. Then, if
  $\lambda^\pm(F,\hat\mu)= 0$, the support of $\hat\mu$ is saturated
  by $\mc{W}^s$ and $\mc{W}^u$.
\end{theorem}

\subsection{Solving the cohomological equation in Lie groups}
\label{sec:Kalinin}

We state in this section two results due to Kalinin
\cite{KalininLivThmMat} which will play an important role in our proof
of Theorem~\ref{thm:lyap-exp-dom-cobound}. However, it is interesting
to remark that our Theorem A is completely independent of those
results of Kalinin and, in fact, the results of
Liv\v{s}ic~\cite{LivCertPropHomol,LivsicCohomDynSys} are enough to
deal with the case where the fibers are one-dimensional.

The first result of Kalinin is the following one:

\begin{theorem}[Theorem 1.4 in \cite{KalininLivThmMat}]
  \label{thm:Kalinin1.4}
  Let $A\colon M\to\GL_d(\bb{R})$ be a H\"older cocycle over a
  transitive hyperbolic homeomorphism $f$. Assume that $A$ satisfies
  the \eqref{eq:POO} condition. Then, the cocycle $A$ has zero
  Lyapunov exponents with respect to any $f$-invariant ergodic measure
  on $M$.
\end{theorem} 

\begin{remark}
  \label{rem:FiberBundleKalinin}
  As it is already observed in \cite{KalininLivThmMat}, the previous
  result does not make use of the fact that the cocycle acts on a
  trivial fiber bundle, and so it holds for arbitrary bundle maps. We
  would like to remark that we can get another proof of
  Theorem~\ref{thm:Kalinin1.4} as consequence of our
  Theorem~\ref{thm:non-zero-exp} below.
\end{remark}

In order to show that the solution of cohomological equations are
sufficiently regular, we must show that the holonomy maps of certain
foliations are smooth. We will show this proving that certain linear
cocycles are indeed coboundaries. As it was already mentioned above,
that can be done invoking ``classical'' Liv\v{s}ic theorem when
cocyles act on one-dimensional manifolds. In higher dimensions, we
need the following Liv\v{s}ic type theorem for linear cocycles due to
Kalinin:

\begin{theorem}[Theorem 1.1 in \cite{KalininLivThmMat}]
  \label{thm:Kalinin1.1}
  Let $A\colon M \to\GL_d(\R)$ a H\"older cocycle over a transitive
  hyperbolic homeomorphism $f\colon M\carr$. Assume that $A$ satisfies
  \eqref{eq:POO} condition. Then, for every $x_0\in M$, there exists a
  unique Hölder map $U\colon M\to\GL_d(\R)$ such that
  \begin{displaymath}
    A(x)=U(f(x))U(x)^{-1}, \quad\forall x\in M,
  \end{displaymath}
  and $U(x_0)=Id_{\R^d}$. 

  Moreover, there exists a constant $C>0$ depending only on $f$ such
  that
  \begin{displaymath}
    \abs{U}_\alpha\leq C\abs{A}_\alpha.
  \end{displaymath}
\end{theorem}

The following consequence of the previous result will be needed later:

\begin{proposition}
  \label{pro:Kalinin-VariacionContinua} Let $f\colon M\carr$ be a
  hyperbolic homeomorphism and $(A_t\colon M \to\GL_d(\R))_{t\in N}$
  be a continuous family ($N$ is a topological manifold) of
  $\alpha$-H\"older cocycles satisfying \eqref{eq:POO} condtion. Then,
  there exists a continuous family of $\alpha$-H\"older transfer
  functions $(U_t\colon M\to\GL_d(\R))_{t\in N}$ satisfying
  \begin{displaymath}
    A_t(x)=U_t(f(x))U_t(x)^{-1},\quad\forall t\in N,\ \forall x\in M. 
  \end{displaymath}
\end{proposition}

\begin{proof} 
  Consider a point $x_0\in M$ whose forward orbit by $f$ is dense and
  let $U_t: M \to\GL_d(\R)$ be the unique solution of the
  cohomological equation $A_t(x) = U_t(f(x))U_t(x)^{-1}$ with
  $U_t(x_0)=Id$, given by Theorem~\ref{thm:Kalinin1.1}. Notice that
  $\abs{U_t}_\alpha$ is uniformly bounded on $t\in N$.

  Now we have to show that $U_t$ depends continuously on $t\in N$. Let
  us fix $t_0\in N$ and $\eps>0$. Let $n$ be sufficiently large so
  that the segment of orbit $x_0, \ldots, f^n(x_0)$ is $\delta$-dense,
  where
  \begin{displaymath}
    \delta:= \eps \left(6C\max_{t\in
        N}\abs{A_t}_\alpha\right)^{-\frac{1}{\alpha}},
  \end{displaymath}
  and $C$ is the positive constant given by
  Theorem~\ref{thm:Kalinin1.1}.

  By continuity of the family $A_t$ on $t$, and observing that
  $U_t(f^k(x_0))=A_t^{(k)}(x_0)$ for every $k\geq 0$, there exists a
  neighborhood $V$ of $t_0$ such that one has
  \begin{displaymath}
    \norm{U_t\big(f^k(x_0)\big) - U_{t_0}\big(f^k(x_0)\big)}
    \leq\frac{\eps}{2},\quad\text{for } 0\leq k\leq n,\ \forall t\in
    V. 
  \end{displaymath}
  
  Invoking the uniform H\"older estimates, we deduce that the
  $C^0$-distance between the functions $U_t$ and $U_{t_0}$ is smaller
  than $\eps$, for every $t\in V$.
\end{proof}

\section{Domination, zero Lyapunov exponents and coboundaries }
\label{sec:domination-Lyap-exp}

In this section we study the relation between domination, nullity of
Lyapunov exponents and cohomology of bundle maps. The main result we
present here is the following

\begin{theorem}
  \label{thm:lyap-exp-dom-cobound}
  Let $N\to\mc{E}\xrightarrow{\pi} M$ be a $C^{\alpha,1}$-fiber bundle
  and $F\colon\mc{E}\carr$ be $C^{\alpha,1}$-bundle map over a
  $\alpha$-H\"older hyperbolic homeomorphism $f\colon M\carr$. Let us
  assume $F$ satisfies \eqref{eq:POO-bund-map} condition. Then the
  following statements are equivalent:
  \begin{enumerate}[(i)]
  \item $\lambda^\pm(F,\hat\mu)=0$, for all $\hat\mu\in\M(F)$;
  \item $F$ is $\alpha$-dominated;
  \item $F$ is a $C^{\alpha,1}$-coboundary and, according to
    \S~\ref{sec:poo-cobound-fiber-maps}, the fiber bundle
    $N\to\mc{E}\xrightarrow{\pi} M$ admits a
    $C^{\alpha,1}$-trivialization.
  \end{enumerate}
\end{theorem}

\begin{remark}
  \label{rem:d-to-d-alpha} 
  We shall use a rather classical trick (see for example
  \cite{VianaAlmostAllCocyc}) which allows us to reduce the general
  $\alpha$-H\"older case to the Lipschitz one: if
  $N\to\mc{E}\xrightarrow{\pi}(M,d)$ is a $C^{\alpha,1}$-fiber bundle,
  $f\colon (M,d)\carr $ is a hyperbolic homeomorphism and
  $F\colon\mc{E}\carr$ is a $C^{\alpha,1}$-fiber bundle map which is
  $\alpha$-dominated, then changing the metric $d$ by $d_\alpha$ (see
  \eqref{eq:alpha-distance}) on $M$, we obtain a $C^{\Lip,1}$-fiber
  bundle, $f$ continues to be hyperbolic (see
  Remark~\ref{rem:change-metric}) and $F$ turns to be a
  $C^{\Lip,1}$-bundle map which is $1$-dominated. Moreover, $F$ is a
  $C^{\Lip,1}$-coboundary when $M$ is endowed with the $d_\alpha$
  metric if and only if it is a $C^{\alpha,1}$-coboundary when $M$ is
  equipped with $d$.
\end{remark}

By Remark~\ref{rem:d-to-d-alpha}, in order to simplify the notation
from now on and until the end of this section we shall assume that
$\alpha=1=\Lip$.

To start with the proof of Theorem~\ref{thm:lyap-exp-dom-cobound}, we
first show that $(i)$ implies $(ii)$. This result maybe belongs to the
folklore, but since our context is slightly different from usual ones,
we decided to include an outline of the proof:

\begin{proposition}
  \label{prop:zero-exp-impl-dom} 
  Let us assume that
  \begin{displaymath}
    \lambda^\pm(F,\hat\mu)=0,\quad\forall\hat\mu\in\M(F),
  \end{displaymath}
  Then $F$ is $1$-dominated.
\end{proposition}

\begin{proof}
  Let us show that $F$ is $(u,1)$-dominated.  The $(s,1)$-domination
  follows from completely analogous arguments.  To prove that, we
  shall only use the hypothesis $\lambda^+(F,\hat \mu)=0$, for every
  $\hat\mu\in\M(F)$.

  Then, let us consider the fiber bundle
  $\pi_{\bb{P}}\colon\bb{P}\to\mc{E}$, where the fiber over an
  arbitrary $\zeta\in\mc{E}$ is given by the projectivized tangent
  space of the submanifold $\mc{E}_{\pi(\zeta)}\subset\mc{E}$.

  Now, the derivative-along-fiber operator $\dpf F$ defined by
  \eqref{eq:dpf-definition} naturally induces a bundle map $[\dpf
  F]\colon\bb{P}\righttoleftarrow$ over
  $F\colon\mc{E}\righttoleftarrow$.

  Then we consider the continuous real cocycle $\psi\colon\bb{P}\to\R$
  over $[\dpf F]$ given by
  \begin{displaymath}
    \psi([v]):=\log\frac{\norm{\dpf F\cdot v}_{F(\zeta)}}{\norm{v}_{\zeta}},
    \quad\forall \zeta\in\mc{E},\ \forall v\in
    T\mc{E}_{\pi(\zeta)}\setminus\{0\}, 
  \end{displaymath}
  where $[v]$ denotes the element of $\bb{P}$ induced by $v$.
  
  Now, let $K_0,\lambda$ be the constants and $\nu_u$ be the
  multiplicative cocycle associated to $f$ given by
  Definition~\ref{def:hyperbolic-homeo}, and suppose $F$ is not
  $(u,1)$-dominated. Then, there exists a sequence of points
  $(\zeta_n)_{n\geq 1}$ in $\mc{E}$ and a strictly increasing sequence
  of natural numbers $(\ell_n)_{n\geq 1}$ such that
  \begin{equation}
    \label{eq:F-not-u-dominated-hypo}
    \norm{\dpf F^{\ell_n}(\zeta_n)}\geq
    \frac{\nu_u^{(\ell_n)}\big(\pi(\zeta_n)\big)}{2},
    \quad\forall n\geq 1. 
  \end{equation}

  This implies that for each $n\in\N$ we can find
  $[v_n]\in\bb{P}_{\zeta_n}$ such that
  \begin{equation}
    \label{eq:v-n-choosing}
    \psi^{(\ell_n)}([v_n])=\sum_{j=0}^{\ell_n-1}\psi\Big([\dpf
    F^j][v_n]\Big) = \log\norm{\dpf F^{\ell_n}(\zeta_n)}.
  \end{equation}

  Then, by Banach-Alaoglu theorem, there is no lost of generality
  assuming that there exists $\tilde\eta\in\M(\bb{P})$ such that
  \begin{equation}
    \label{eq:seq-measures-hat-eta}
    \frac{1}{\ell_n}\sum_{j=0}^{\ell_n-1}
    [\dpf F^j]_\star(\delta_{[v_n]})\to\tilde\eta, \quad\text{as }
    n\to\infty, 
  \end{equation}
  where the convergence is in the weak-$\star$ topology.

  Putting together \eqref{eq:F-not-u-dominated-hypo},
  \eqref{eq:v-n-choosing} and \eqref{eq:seq-measures-hat-eta}, we can
  easily show that 
  \begin{equation}
    \label{eq:tilde-eta-estimate}
    \int_{\bb{P}}\psi\dd\tilde\eta=\lim_{n\to\infty}
    \frac{\psi^{(\ell_n)}([v_n])}{\ell_n}\geq
    \lim_{n\to\infty}\frac{1}{\ell_n}\Big(\log K_0 +
    \lambda\ell_n-2\Big)=\lambda. 
  \end{equation}

  Finally, defining $\eta:={\pi_{\bb{P}}}_\star(\tilde\eta)$, we get
  $\eta\in\M(F)$ and from \eqref{eq:tilde-eta-estimate} it easily
  follows
  \begin{displaymath}
    \lambda^+(F,\eta)\geq\lambda>0,
  \end{displaymath}
  contradicting our hypothesis.
\end{proof}

Next we show that $(ii)$ implies $(iii)$ in
Theorem~\ref{thm:lyap-exp-dom-cobound}. 

Since we are assuming $F$ is $1$-dominated, by
Proposition~\ref{prop:ConosVariedadesInestables} we know that we can
lift the stable and unstable sets of $f$ to $\mc{E}$. We shall need
the following properties of them:
\begin{lemma}
  \label{lem:ConosVariedadesInestables}
  If $\mc{W}^\sigma$ denotes the lift of $W^\sigma$ (with
  $\sigma\in\{s,u\}$), then there exists a constant $K \geq 1$ such
  that
  \begin{displaymath}
    d_{\mc{E}}(\zeta,\eta) \leq K
    d\big(\pi(\zeta),\pi(\eta)\big), 
  \end{displaymath}
  for every $\zeta\in\mc{E}$, $\eta\in\mc{W}^\sigma(\zeta)$ and such
  that $\pi(\eta)\in W^\sigma_{\delta_0}(\pi(\zeta),f)$, where
  $\delta_0$ is the constant associated to $f$ by
  Definition~\ref{def:hyperbolic-homeo}.
\end{lemma}

\begin{proof}
  This is a straightforward consequence of
  Proposition~\ref{prop:ConosVariedadesInestables}, \ie the fact that
  the elements of $\mc{W}^\sigma$ are graphs of Lipschitz functions
  with uniformly bounded constant over the stable and unstable sets of
  $f$.
\end{proof}

Then we need the following result that plays a key role in the
construction of solutions for the cohomological equation:
\begin{proposition}
  \label{pro:PartialHypImplicaGraficos}   
  If $F$ is $1$-dominated and satisfies \eqref{eq:POO-bund-map}
  condition, then the closure of every $F$-orbit is the image of a
  Lipschitz section.  More precisely, for every $\zeta\in\mc{E}$,
  there exists a Lipschitz section
  $V_\zeta\colon\overline{\mc{O}_f\big(\pi(\zeta)\big)}\subset M\to
  \mc{E}$ such that
  \begin{displaymath}
    \overline{\mc{O}_F(\zeta)}=\left\{ V_\zeta(y)\in\mc{E} :
      y\in\overline{\mc{O}_f\big(\pi(\zeta)\big)}\right\}.  
  \end{displaymath}
\end{proposition}

\begin{proof}
  In order to show that the closure of any $F$-orbit coincides with
  the image of a continuous section of the fiber bundle
  $N\to\mc{E}\xrightarrow{\pi} M$, it is enough to show the following
  \begin{claim}
    \label{claim:return}
    For every $\zeta\in\mc{E}$ and every $\varepsilon>0$, there exists
    $\delta>0$ such that
    \begin{displaymath}
      d_{\mc{E}}\big(\zeta,F^{n}(\zeta)\big)<\varepsilon,
    \end{displaymath}
    whenever $d\big(\pi(\zeta),f^n(\pi(\zeta))\big)<\delta$.
  \end{claim}

  To prove Claim~\ref{claim:return}, let $\zeta\in\mc{E}$ and
  $\varepsilon>0$ be arbitrary. Then, let us choose
  $\delta:=\min\big(\delta_0,\delta_1,\varepsilon(4cK)^{-1}\big)$,
  where constants $\delta_1$ and $c$ are given by
  Theorem~\ref{thm:AnosovClosing} and $K$ is given by
  Lemma~\ref{lem:ConosVariedadesInestables}.

  Then, suppose $n\in\N$ is given such that
  $d\big(\pi(\zeta),f^n(\pi(\zeta))\big)<\delta$. Since $f$ is a
  hyperbolic homeomorphism and $\delta\leq\delta_1$, we can apply
  Theorem~\ref{thm:AnosovClosing} to guaranty the existence of
  $p\in\Per(f)$ and $y:=\big[\pi(\zeta),p]\in M$ satisfying (1), (2)
  and (3) in Theorem~\ref{thm:AnosovClosing}. Thus, taking into
  account that the fiber bundle projection $\pi$ is one-to-one on
  $\mc{W}^u(\zeta)\subset\mc{E}$ and $y\in
  W^u(\pi(\zeta),f)=\pi\big(\mc{W}^u(\zeta)\big)$, there exists a
  unique point $\zeta_y\in\mc{E}_y\cap\mc{W}^u(\zeta)$. Analogously,
  $\pi$ is one-to-one from $\mc{W}^s(\zeta_y)$ onto $W^s(y)$ and
  hence, there exists a unique point
  $\zeta_p\in\mc{E}_p\cap\mc{W}^s(\zeta_y)$.

  Now, observing that $\zeta_p\in\mc{W}^s(\zeta_y)$ and
  $\zeta_y\in\mc{W}^u(\zeta)$, we can combine
  Theorem~\ref{thm:AnosovClosing} and
  Lemma~\ref{lem:ConosVariedadesInestables} to guaranty that
  \begin{equation}
    \label{eq:zeta-init-zeta-p}
    d_{\mc{E}}(\zeta,\zeta_p)\leq
    d_{\mc{E}}(\zeta,\zeta_y)+d_{\mc{E}}(\zeta_y,\zeta_p)\leq
    K\Big[d\big(\pi(\zeta),y\big) + d(y,p)\Big]<
    2Kc\delta\leq\frac{\varepsilon}{2},
  \end{equation}
  and
  \begin{equation}
    \label{eq:zeta-returning-zeta-p}
    \begin{split}
      d_{\mc{E}}\big(F^n(\zeta),F^n(\zeta_p)\big)&\leq
      d_{\mc{E}}\big(F^n(\zeta),F^n(\zeta_y)\big) +
      d_{\mc{E}}\big(F^n(\zeta_y),F^n(\zeta_p)\big) \\
      & \leq K\Big[d\big(f^n(\pi(\zeta),f^n(y)\big) +
      d\big(f^n(y),f^n(p)\big)\Big] \\
      &\leq 2Kcd\big(f^n(\pi(\zeta)),\pi(\zeta) \big) < 2Kc\delta\leq
      \frac{\varepsilon}{2},
    \end{split}
  \end{equation}

  Finally observe that, since $f^n(p)=p$ and $F$ satisfies
  \eqref{eq:POO-bund-map}, it holds $F^n(\zeta_p)=\zeta_p$. Then,
  putting together \eqref{eq:zeta-init-zeta-p} and
  \eqref{eq:zeta-returning-zeta-p}, we get
  $d_{\mc{E}}(\zeta,F^n(\zeta))<\varepsilon$, and our claim is proven.
\end{proof}

It is interesting to notice that the exponential shadowing given by
Anosov closing lemma (Theorem~\ref{thm:AnosovClosing}) was not used in
the proof of Proposition~\ref{pro:PartialHypImplicaGraficos}. In fact,
the classical Shadowing Lemma is enough because the H\"older
regularity (in this case Lipschitz) was already used in
Lemma~\ref{lem:ConosVariedadesInestables}.

Now, let us fix any point $x_0\in M$ such that its forward and
backward $f$-orbits are dense. Then, by
Proposition~\ref{pro:PartialHypImplicaGraficos}, assuming $F$ is
$1$-dominated, for every $\zeta\in\mc{E}_{x_0}$ there exists a
continuous section $V_\zeta\colon M\to\mc{E}$ such that
$\overline{\mc{O}_F(\zeta)}$ coincides with the image of $V_\zeta$. To
simplify the notation, the image of the section $V_\zeta$ will be
denoted by $\mc{V}_\zeta$, \ie we define $\mc{V}_\zeta:=\{V_\zeta(x)
\in\mc{E} : x\in M\}$, for every $\zeta\in\mc{E}_{x_0}$.

Then we will show that the family
$\left\{\mc{V}_\zeta\right\}_{\zeta\in\mc{E}_{x_0}}$ determines a
continuous lamination in $\mc{E}$. To do this, we first prove the
following

\begin{proposition}
  \label{pro:su-saturation-V-zeta}
  Assuming $F$ is $1$-dominated, for each $\zeta\in\mc{E}_{x_0}$ the
  image of the section $V_\zeta$ defined above is saturated by leaves
  of the lamination $\mc{W}^s$ \emph{($\mc{W}^u$, respectively.)} More
  precisely, for every $\zeta\in\mc{E}_{x_0}$ and any
  $\eta\in\mc{V}_\zeta$,
  \begin{displaymath}
    \mc{W}^\sigma(\eta)\subset\mc{V}_\zeta,\quad\text{for }
    \sigma\in\{s,u\}.
  \end{displaymath}
\end{proposition}

\begin{proof}
  Let us suppose the proposition is not true. Then, there exists some
  $\zeta\in\mc{E}_{x_0}$ and $\eta^\prime\in\mc{V}_\zeta$ such that
  $\mc{W}^s(\eta^\prime)\not\subset\mc{V}_\zeta$
  \emph{($\mc{W}^u(\eta^\prime)\not\subset\mc{V}_\zeta$,
    respectively).} By continuity of the section $V_\zeta$ and the
  stable \emph{(unstable, respec.)} lamination, we can choose a point
  $\eta\in\mc{V}_\zeta$ such that the forward \emph{(backward,
    respec.)} $f$-orbit of $\pi(\eta)$ is dense in $M$ and
  $\mc{W}^s(\eta)\not\subset\mc{V}_\zeta$
  \emph{($\mc{W}^u(\eta)\not\subset\mc{V}_\zeta$, respec.)} Then, we
  take a point $\xi\in\mc{W}^s(\eta)\setminus\mc{V}_\zeta$
  \emph{($\xi\in\mc{W}^u(\eta)\setminus\mc{V}_\zeta$, respec.).}
  Observe that $\mc{O}_f^+\big(\pi(\xi)\big)$ \emph{(
    $\mc{O}_f^-\big(\pi(\xi)\big)$, respec.)} is dense in $M$. Hence,
  the section $V_\xi$ given by
  Proposition~\ref{pro:PartialHypImplicaGraficos} is defined on the
  whole space $M$. But, since $\xi\in\mc{W}^s(\eta)$
  \emph{($\xi\in\mc{W}^u(\eta)$, respec.)}, the set
  $\overline{\mc{O}_F(\xi)}$ intersects the fiber $\mc{E}_{\pi(\xi)}$
  at two different points: at $\xi$ and at
  $V_\zeta\big(\pi(\xi)\big)$, contradicting
  Proposition~\ref{pro:PartialHypImplicaGraficos}.
\end{proof}

\begin{remark}
  It is interesting to notice that a less elementary proof of
  Proposition~\ref{pro:su-saturation-V-zeta} can be easily gotten
  invoking the topological version of the Invariance Principle of
  Avila and Viana (see Theorem \ref{thm:inv-principle}). In fact,
  assuming domination and \eqref{eq:POO-bund-map}, using
  Theorem~\ref{thm:Kalinin1.4}, it can be shown that condition $(i)$
  of Theorem~\ref{thm:lyap-exp-dom-cobound} holds and then the
  Invariance Principle can be applied.
\end{remark}

As a consequence of Proposition~\ref{pro:su-saturation-V-zeta} we get
the family $\left\{\mc{V}_\zeta\right\}_{\zeta\in\mc{E}_{x_0}}$ is a
partition of the total space $\mc{E}$, and moreover, a continuous
lamination whose leaves are (topologically) transverse to the fibers
of the fiber bundle $N\to\mc{E}\xrightarrow{\pi}M$. Thus, we can
define the \emph{holonomy maps} of this lamination as follows: given
arbitrary points $x,y\in M$, we define the \emph{holonomy map from $x$
  to $y$} is defined by
\begin{equation}
  \label{eq:Hol-definition}
  \Hol_{x,y} : \mc{E}_x\ni \zeta \mapsto V_{\hat\zeta}(y)\in\mc{E}_y,
\end{equation}
where $\hat\zeta$ is the unique point in $\mc{E}_{x_0}$ such that
$\zeta\in\mc{V}_{\hat\zeta}$. Observe that, by
Proposition~\ref{pro:su-saturation-V-zeta}, the holonomy maps are (at
least) homeomorphisms. After some additional results, we shall show
they are indeed $C^1$-diffeomorphisms.

Then we get the following

\begin{proposition}
  \label{pro:H-homeo-definition}
  The fiber bundle map $F\colon\mc{E}\carr$ is a
  $C^{\Lip,0}$-coboundary. More precisely, the fiber bundle
  $N\to\mc{E}\xrightarrow{\pi}M$ admits a continuous trivialization
  $H\colon\mc{E}\to M\times N$ that turns the
  diagram~\eqref{eq:diag-cobound-trivial} commutative.
\end{proposition}

\begin{proof}
  To show that the fiber bundle is trivial, let us consider the map
  $H\colon\mc{E}\to M\times N$ given by
  \begin{equation}
    \label{eq:H-definition}
    H(\zeta):=\bigg(\pi(\zeta),
    \mathrm{pr_2}
    \Big(\phi_j\big(\Hol_{\pi(\zeta),x_0}(\zeta)\big)\Big)\bigg), 
    \quad\forall\zeta\in\mc{E},
  \end{equation}
  where $\phi_j\colon U_j\to M\times N$ is any (but fixed)
  trivializing chart of the fiber bundle $N\to\mc{E}\xrightarrow{\pi}
  M$ such that $x_0\in U_j$. Then, since holonomy maps are
  homeomorphisms, it is clear that $H$ itself is a homeomorphism, and
  since $F\big(\mc{V}_{\hat \zeta}\big)=\mc{V}_{\hat \zeta}$, for
  every $\hat\zeta\in\mc{E}_{x_0}$, we conclude that
  \begin{displaymath}
    H\big(F(\zeta)\big) = (f\times id_N)\big(H(\zeta)\big), \quad\forall
    \zeta\in\mc{E},
  \end{displaymath}
  as desired.
\end{proof}

Finally, in order to show that $F$ is a $C^{\Lip,1}$-coboundary, it
remains to prove that the map $H\colon\mc{E}\to M\times N$ constructed
in the proof of Proposition~\ref{pro:H-homeo-definition} is indeed a
$C^{\Lip,1}$-bundle map.

To do this, it is necessary to show that the holonomy maps defined in
\eqref{eq:Hol-definition} are differentiable and this will be gotten
invoking Proposition~\ref{pro:Kalinin-VariacionContinua}. To use this
result, we first need the following
\begin{lemma}
  \label{lem:V-zeta-holder-reg}
  For every $\zeta\in\mc{E}_{x_0}$, the section $V_{\zeta}\colon
  M\to\mc{E}$ (whose image is $\mc{V}_{\zeta}$) is Lipschitz.
\end{lemma}

\begin{proof}
  It is a straightforward consequence of the fact that the graph of
  $V_\zeta$ is saturated by $\mc{W}^s$ and $\mc{W}^u$, which are
  Lischitz themselves and have local product structure (see (h5) in
  Definition~\ref{def:hyperbolic-homeo}).
\end{proof}

Now, for every $\zeta\in\mc{E}_{x_0}$, consider the set 
\begin{displaymath}
  \Xi^{\zeta}:=\bigsqcup_{x\in M} T_{V_\zeta(x)}\mc{E}_{x},
\end{displaymath}
where $\bigsqcup$ denotes the disjoint-union operator, and the
``natural projection'' map $\pi^\zeta\colon\Xi^\zeta\to M$ given by
$(\pi^\zeta)^{-1}(x)=T_{V_\zeta(x)}\mc{E}_{x}$, for every $x\in M$. By
Lemma~\ref{lem:V-zeta-holder-reg}, the set $\Xi^\zeta$ can naturally
be endowed with an appropriate vector bundle structure turning
$\R^d\to\Xi^\zeta\xrightarrow{\pi^\zeta}M$ into a $C^{\Lip}$-vector
bundle, where $d=\dim N$.

On the other hand, since every leaf $\mc{V}_\zeta$ is $F$-invariant
and $F\big|_{\mc{E}_{x}}\colon\mc{E}_{x}\to\mc{E}_{f(x)}$ is a
$C^1$-diffeomorphism, our fiber bundle map $F$ naturally induces a
$C^{\Lip}$-vector bundle map $DF^\zeta\colon\Xi^\zeta\carr$ over
$f\colon M\carr$ given by
\begin{equation}
  \label{eq:Df-zeta-def}
  DF^\zeta(v_x)=\dpf F\big(V_\zeta(x)\big)(v_x),
  \quad\forall x\in M,\ \forall v_x\in
  \Xi^\zeta_x=T_{V_\zeta(x)}\mc{E}_{x}, 
\end{equation}
where $\dpf F$ denotes the (partial) derivative along the fibers
defined in \S~\ref{sec:lyap-expon-fiber-bundle-maps}.

Then we get the following
\begin{proposition}
  \label{pro:transverse-deriv-cocycle}
  For every $\zeta\in\mc{E}_{x_0}$, the vector bundle
  $\R^d\to\Xi^\zeta\xrightarrow{\pi^\zeta} M$ is trivial and the
  vector bundle map $DF^\zeta$ is a $C^{\Lip}$-coboundary, \ie there
  exists a $C^{\Lip}$-vector bundle map $U^\zeta\colon\Xi^\zeta\to
  M\times\R^d$ satisfying
  \begin{displaymath}
    U^\zeta\circ DF^\zeta=(f\times Id_{\R^d})\circ U^\zeta.
  \end{displaymath}
  
  Moreover, the family $(U^\zeta)_{\zeta\in\mc{E}_{x_0}}$ can be
  chosen to vary continuously on $\zeta$.
\end{proposition}

\begin{proof}
  Since $F$ satisfies \eqref{eq:POO-bund-map} condition, $DF^\zeta$
  must satisfy it, too. Hence, by Theorem~\ref{thm:Kalinin1.4},
  $DF^\zeta$ has zero Lyapunov exponents with respect to any
  $f$-invariant probability measure. In particular, invoking
  Proposition~\ref{pro:H-homeo-definition} we conclude that the vector
  bundle $\R^d\to\Xi^\zeta\xrightarrow{\pi^\zeta} M$ is trivial and we
  can apply Proposition~\ref{pro:Kalinin-VariacionContinua} to obtain
  a continuous family $U^\zeta$ of solutions, as desired.
\end{proof}

Then we get the following 
\begin{corollary}
  \label{cor:DFn-uniform-estimate} 
  If $F$ is $1$-dominated and satisfies \eqref{eq:POO-bund-map}, then
  there exists $C>0$ such that
  \begin{displaymath}
    \norm{\partial_{\mathrm{fib}}F^n(v)}\leq C, \quad\forall
    n\in\Z,\ \forall \zeta\in\mc{E},\ \forall v\in
    T_\zeta\mc{E_{\pi(\zeta)}}. 
  \end{displaymath}
\end{corollary}

\begin{proof}
  This is a straightforward consequence of
  Propositions~\ref{pro:transverse-deriv-cocycle} and
  \ref{pro:Kalinin-VariacionContinua}.
\end{proof}

Then we finally get
\begin{proposition}
  \label{pro:H-conjugacy-smoothness}
  The holonomy maps given by \eqref{eq:Hol-definition} are
  differentiable and consequently, the map $H\colon\mc{E}\to M\times
  N$ defined by \eqref{eq:H-definition} is a $C^{\Lip,1}$-bundle map.
\end{proposition}

\begin{proof}
  Given arbitrary points $x,y\in M$, we need to show that the holonomy
  map $\Hol_{x,y}\colon\mc{E}_x\to\mc{E}_y$ associated to the
  lamination $\{\mc{V}\}_{\zeta\in\mc{E}_{x_0}}$, which is clearly a
  homeomorphism, is indeed a $C^1$-diffeomorphism.

  To do this, first observe that since each leaf of the lamination
  $\{\mc{V}_\zeta\}_{\zeta\in\mc{E}_{x_0}}$ is $F$-invariant, it holds
  \begin{equation}
    \label{eq:Hol-points-same-orb}
    \Hol_{f^m(x),f^n(x)}=F^{n-m}\big|_{\mc{E}_{f^m(x)}}, \quad\forall
    x\in\mc{E},\ \forall m,n\in\Z.
  \end{equation}
  Consequently, holonomy maps between any two points of the same
  $f$-orbit are indeed $C^1$-diffeomorphisms.

  To deal with the general case, let $x,y\in M$ be arbitrary points
  and consider two trivializing charts $\varphi_i\colon
  \pi^{-1}(U_i)\to U_i\times N$, with $i=1,2$, such that $x\in U_1$
  and $y\in U_2$. Recalling we have chosen $x_0\in M$ so that its
  forward $f$-orbit is dense in $M$, we can find two sequences of
  natural numbers $(m_i)$ and $(n_i)$ such that $U_1\ni
  f^{m_i}(x_0)\to x$ and $U_2\ni f^{n_i}(x_0)\to y$, as
  $i\to\infty$.

  Then, for each $i\geq 1$, let us define $\Hol_i\in\Diff{}1(N)$ by
  \begin{displaymath}
    \Hol_i(p):= \mathrm{pr}_2\circ \varphi_2\circ
    \Hol_{f^{m_i}(x_0),f^{n_i}(x_0)}
    \circ\varphi_1^{-1}\big(f^{m_i}(x_0),p\big), 
  \end{displaymath}
  and $\Hol\in\Homeo{}(N)$ by
  \begin{displaymath}
    \Hol(p)=\mathrm{pr}_2\circ \varphi_2\circ
    \Hol_{x,y}\circ\varphi_1^{-1}\big(x,p\big),
  \end{displaymath}
  for every $p\in N$. We want to show $\Hol\in\Diff{}1(N)$, too.

  By continuity of the lamination
  $\{\mc{V}_\zeta\}_{\zeta\in\mc{E}_{x_0}}$, when $i\to\infty$,
  $\Hol_i\to\Hol$ pointwisely. By
  Corollary~\ref{cor:DFn-uniform-estimate} and Arzel\`a-Ascoli
  theorem, we conclude $\Hol_i\to\Hol$ $C^0$-uniformly. On the other
  hand, by Propositions~\ref{pro:transverse-deriv-cocycle} and
  \ref{pro:Kalinin-VariacionContinua}, we have the sequence of
  derivatives $(D\Hol_i(p))_{i\geq 1}$ is also convergent, for each
  $p\in N$. Consequently, $\Hol$ is $C^1$ and then,
  $\Hol_{x,y}\colon\mc{E}_x\to\mc{E}_y$ is a diffeomorphism, as
  desired.
\end{proof}

Finally, it remains to show that $(iii)$ implies $(i)$ in
Theorem~\ref{thm:lyap-exp-dom-cobound}.  But this is obvious, because
a $C^{\Lip,1}$-coboundary is, by definition, conjugate to the map
$(f\times id_N)\colon M\times N\carr$ via a $C^{\Lip,1}$-fiber bundle
conjugacy, and therefore, every Lyapunov exponent must vanish.

\section{Domination as a consequence of POO condition}
\label{sec:POO-implies-zero-Lyap}

In this section we shall review some contexts where condition
\eqref{eq:POO} alone implies that the cocycle is dominated, and as a
consequence of Theorem~\ref{thm:lyap-exp-dom-cobound}, it is a
coboundary.

We start proving Theorem B which follows from
Theorem~\ref{thm:lyap-exp-dom-cobound} and Katok closing lemma
\cite{KatokLyapExp}:

\begin{proof}[Proof of Theorem B]
  Let $S\to\mc{E}=M\times S\xrightarrow{\pi} M$ denote the trivial
  fiber bundle and $F\colon M\times S\carr$ be the $C^{1+\alpha}$
  skew-product over $f$ induced by $\Phi$ as in
  \eqref{eq:induced-skew-product}. Oberserve that $F$ satisfies
  \eqref{eq:POO-bund-map}.

  Let us suppose there exists an $F$-invariant ergodic probability
  measure $\hat\mu$ such that $\lambda^+(F,\hat\mu)\neq 0$. 

  Since $\Phi$ takes values in the group of area-preserving
  diffeomorphisms of $S$, by Oseledets theorem we know that
  \begin{displaymath}
    \label{eq:lambda+-lambda-area-preserving}
    \lambda^-(F,\hat\mu)+\lambda^+(F,\hat\mu)=0.
  \end{displaymath}
  So, we have 
  \begin{displaymath}
    \lambda^-(F,\hat\mu)<0<\lambda^+(F,\hat\mu),
  \end{displaymath}
  and since $f\colon M\carr$ is an Anosov diffeomorphism, this implies
  $\hat\mu$ is a hyperbolic measure for $F$ (\ie all its Lyapunov
  exponents given by Oseledets theorem are different from zero).

  So, applying Katok closing lemma \cite[Corollary 4.3]{KatokLyapExp},
  we conclude that $F$ exhibits a hyperbolic periodic point. But, by
  \eqref{eq:POO-bund-map}, if $\zeta_0\in\mc{E}$ is periodic with
  $F^n(\zeta_0)=\zeta_0$, then $F^n(\zeta)=\zeta$, for every
  $\zeta\in\mc{E}_{\pi{\zeta_0}}$, and so $\zeta_0$ is not hyperbolic,
  getting a contradiction.
\end{proof}

The amount of regularity required in the fiber direction is essential
in our argument and it is the usual one in Pesin's theory which allows
to obtain a subexponential neighborhoods of a regular orbit with good
estimates on the bundles of the Oseledet's splitting (see
\cite[Suplement]{KatokHasselblatt}). The recent examples of
\cite{BonattiCrovisierShinohara} show that improving this regularity
requires new ideas which should not be different from the general case
of $\Diff{}{1}(N)$ cocycles with arbitrary $N$.

In order to prove Theorem A, we need the following result that should
be considered the main one of this section:
\begin{theorem}
  \label{thm:non-zero-exp}
  Let $N\to\mc{E}\xrightarrow{\pi}M$ be a $C^{\alpha,1}$-fiber bundle,
  $F\colon\mc{E}\carr$ be a $C^{\alpha,1}$-bundle map over an
  $\alpha$-H\"older hyperbolic homeomorphism $f\colon M\carr$.  If
  there exists an ergodic measure $\hat\mu\in\M(F)$ with
  $\lambda^+(F,\hat\mu)<0$, then there exists $\zeta_0\in\Per(F)$ which
  is uniformly contracting along the fiber, \ie if $n>0$ denotes the
  period of $\zeta_0$, then all the eigenvalues of the linear map
  $\dpf F^n(\zeta_0)\colon T_{\zeta_0}\mc{E}_{\pi(\zeta_0)}\carr$ have
  modulus strictly smaller than $1$.
\end{theorem}

It is interesting to remark that applying
Theorem~\ref{thm:non-zero-exp} to the natural action induced by a
linear cocycle on a suitable Grasmannian fiber bundle (corresponding
to the dimension of the subspace with largest Lyapunov exponent), one
can reprove Kalinin's result on approximation\footnote{The statement
  above only allows to show that if $\Phi\colon M\to\GL_d(\R)$ is an
  $\alpha$-H\"older cocycle satisfying the \eqref{eq:POO} then every
  measure has zero Lyapunov exponents, however, the results below
  allow to recover the complete result if desired.} of Lyapunov
exponents \cite[Theorem 1.4]{KalininLivThmMat}.

Now, we can prove Theorem A as a combination of Theorems
\ref{thm:lyap-exp-dom-cobound} and \ref{thm:non-zero-exp}:

\begin{proof}[Proof of Theorem A]
  Let $\R/\Z\to\mc{E}=M\times\R/\Z\xrightarrow{\pi} M$ denote the
  trivial fiber bundle and $F\colon M\times\R/\Z\carr$ be the
  skew-product over $f$ induced by $\Phi$ as in
  \eqref{eq:induced-skew-product}. Since $\Phi$ satisfies
  \eqref{eq:POO}, $F$ satisfies \eqref{eq:POO-bund-map}.

  Hence, for every $\zeta\in\Per(F)$ with $F^n(\zeta)=\zeta$ it
  clearly holds $\partial_{\mathrm{fib}}F^n=D\Phi^{(n)}\equiv id$ and
  consequently, all the eigenvalues are equal to $1$. So, applying
  Theorem~\ref{thm:non-zero-exp} to $F$ and $F^{-1}$ we get
  \begin{displaymath}
    -\lambda^+(F^{-1},\hat\mu)=\lambda^-(F,\mu)\leq
    0\leq\lambda^+(F,\hat\mu). 
  \end{displaymath}
  But since the fibers are one-dimensional, we can apply Birkhoff
  ergodic theorem to conclude that
  $\lambda^-(F,\hat\mu)=\lambda^+(F,\hat\mu)$. Consequently,
  $\lambda^-(F,\hat\mu)=\lambda^+(F,\hat\mu)=0$ and by
  Theorem~\ref{thm:lyap-exp-dom-cobound}, $F$ is a
  $C^{\alpha,1}$-coboundary, as desired.
\end{proof}

\subsection{Proof of Theorem \ref{thm:non-zero-exp}}

From the uniform continuity of $\dpf F$ and $f$, it easily follows
\begin{lemma}
  \label{Lemma-UniformContinuity}
  For every $\delta>0$, there exists $\chi>0$ such that for every
  $\eta,\xi\in\mc{E}$ satisfying
  \begin{displaymath}
    d_{\mc{E}}\big( F^{i}(\eta), F^{i}(\xi)\big) \leq\chi,\quad
    \text{for every } i\in\{0,\ldots k\},
  \end{displaymath}
 then one has
 \begin{displaymath}
   \prod_{i=0}^{k-1} \norm{\dpf F\big( F^i(\eta)\big)} \leq
   e^{k\delta}  \prod_{i=0}^{k-1} \norm{\dpf F\big(F^i(\xi)\big)}. 
 \end{displaymath}
\end{lemma}

Along the proof we shall assume that
$\lambda^+:=\lambda^+(F,\hat\mu)<0$.

It is a classical fact that one can choose measurable adapted
metrics which see the contraction at each iterate (see for example
Proposition 8.2 of \cite{AbdenurBonattiCrovisierNonunif}):

\begin{lemma}\label{lem:changemetric}
For every $\eps>0$ there exists an integer $N>0$ and a measurable
function $A: \mc{E} \to [1,+\infty)$ such that:
\begin{itemize}
\item The sequence $\big(A(F^n(\zeta)\big)_{n\in\mathbb{Z}}$
varies sub-exponentially (\ie one has that for $\hat\mu$-almost
every $\zeta \in \mc{E}$ the sequence $\frac{1}{|n|} \log
|A(F^n(\zeta))|$ converges to $0$ as $|n|\to \infty$).

\item If we denote $\norm{ \cdot }_{\zeta}'$ to the metric in
$T_{\zeta} \mc{E}_{\pi(\zeta)}$ defined as:
\begin{displaymath}
\norm{v }_{\zeta}' = \sum_{0 \leq k \leq N} e^{-k (\lambda^+ +
\eps)} A(F^k(\zeta)) \norm{\dpf F^k(\zeta) \cdot v }_{F^k(\zeta)}
\end{displaymath}
then, for $\hat \mu$ almost every $\zeta \in \mc{E}$ and every $v
\in T_\zeta \mc{E}_{\pi(\zeta)}$ one has that
\begin{displaymath}
\norm{\dpf F (\zeta) \cdot v}'_{F(\zeta)} \leq e^{(\lambda^+ +
\eps)} \norm{v}'_{\zeta}
\end{displaymath}
\end{itemize}
\end{lemma}

We shall fix $\eps < \min \{-\frac{\lambda^+}{5}, \frac{\alpha
  \lambda}{2}\}$, where $\alpha$ is the H\"{o}lder exponent of $F$ and
$\lambda$ is the hyperbolicity constant appearing in Theorem
\ref{thm:AnosovClosing}. Consider the function $A$ and the metric
$\norm{\cdot}^\prime$ given by the previous lemma and let us fix them
from now on.

Using this metric, it is also standard to show that one can define
sub-exponential neighborhoods (sometimes called \emph{Pesin charts})
of typical points with respect to $\hat \mu$ such that the dynamics in
those neighborhoods behaves similarly to the derivative (see for
example \cite[Supplement]{KatokHasselblatt}).

For $\zeta \in \mc{E}$ we shall consider the exponential map
$\mathrm{exp}: T_{\zeta}\mc{E}_{\pi(\zeta)} \to \mc{E}_{\pi(\zeta)}$
where the distances in $T_\zeta \mc{E}_{\pi(\zeta)}$ are measured with
respect to the metric $\|\cdot \|'_\zeta$. We denote by $B_\zeta'(r)$
the ball of radius $r$ centered at $0$ in $T_\zeta
\mc{E}_{\pi(\zeta)}$.

\begin{lemma}\label{lem:pesincharts}
  There exists a measurable function $\rho: \mc{E} \to (0,+\infty)$
  such that if $\varphi_\zeta= \mathrm{exp}|_{B_\zeta'(\rho(\zeta))}$
  then we have that for $\hat \mu$-almost every point one has that
  $\varphi_{F(\zeta)}^{-1} \circ F \circ \varphi_{\zeta}:
  B_\zeta'(\rho(\zeta)) \to B_{F(\zeta)}'(\rho(F(\zeta)))$ contracts
  vectors by a factor of at least $e^{(\lambda^+ + 2\eps)}$. Moreover,
  it holds that the sequence $\rho(F^n(\zeta))$ is sub-exponential and
  can be chosen so that $e^{-\eps}\rho(\zeta) <\rho(F(\zeta)) < e^\eps
  \rho(\zeta)$.
\end{lemma}

It is relevant to remark here the fact that the sub-exponential growth
of the function $\rho$ is essential in Pesin's theory and it is where
the H\"{o}lder regularity of the derivative is usually used. Here,
since we are working with measures whose Lyapunov exponents are all
negative, $C^1$-regularity in the fibers is enough.

\begin{remark}
  Notice that there is a measurable function $D$ which, associates to
  each $\zeta \in \mc{E}$ an isometry $D_\zeta: \big(T_\zeta
  \mc{E}_{\pi(\zeta)}, \norm{\cdot}_\zeta\big) \to \big(T_\zeta
  \mc{E}_{\pi(\zeta)}, \norm{\cdot}'_\zeta \big)$. When this linear
  map is seen as a transformation from $\big( T_\zeta
  \mc{E}_{\pi(\zeta)},\norm{\cdot}_\zeta\big)$ to itself, the norm and
  co-norm of $D_\zeta$ are bounded by a number depending only on
  $A(\zeta)$.
\end{remark}

Using Luisin's Theorem on approximation of measurable functions by
continuous ones (see for example \cite[Supplement]{KatokHasselblatt})
one obtains a compact set $X \subset \mc{E}$ of positive $\hat
\mu$-measure such that functions $A$ and $\rho$ are continuous on $X$
and, thus, bounded (we define $A_X:=\sup_{\zeta\in X} A(\zeta)$ and
$\rho_X:=\inf_{\zeta\in X}\rho(x)$).

Consider a point $\zeta_0 \in X$ which is recurrent inside $X$, \ie
there exists $n_j \to \infty$ such that $F^{n_j}(\zeta_0) \to \zeta_0$
and $F^{n_j}(\zeta_0) \in X$ for every $j>0$.

Let us write $x_0 := \pi(\zeta_0)$ and, for each $j>0$, let $p_j \in
\Fix(f^{n_j})$ be the periodic point of $f$ given by Anosov Closing
Lemma (Theorem \ref{thm:AnosovClosing}). One has that:
\begin{displaymath}
  d(f^i(x_0), f^i(p_j)) \leq c e^{-\lambda \min \{i,n_j-i\}}
  d(f^{n_j}(x_0),x_0), \quad\text{for } i=0, \ldots, n_j.
\end{displaymath}
where $c,\lambda>0$ are the constants given in Theorem
\ref{thm:AnosovClosing}. Notice that $d(f^{n_j}(x_0),x_0) \leq
d_{\mc{E}} (F^{n_j}(\zeta_0),\zeta_0) \to 0$.

Fix $\delta< \eps$ and let $\chi$ be the constant given by Lemma
\ref{Lemma-UniformContinuity} for such a $\delta$.

The main step in the proof is the following
\begin{lemma}
  \label{Lemma-PuntoCerca}
  For $n_j$ large enough, there exists a small open ball
  $B_j\subset\mc{E}_{p_j}$ such that
  \begin{equation}
    \label{eq:Bj-recurrence}
    F^{n_j}(\overline{B_j})\subset B_j.
  \end{equation}
  Moreover, $\mathrm{diam}(F^i(B_j))\leq \chi/2$ and
  $d_{\mc{E}}(F^i(B_j), F^i(\zeta_0))\leq\chi/2$, for every
  $0\leq i\leq n_j$.
\end{lemma}

Let us now conclude the proof of Theorem~\ref{thm:non-zero-exp}
assuming this lemma:

By \eqref{eq:Bj-recurrence}, we know there exists $\xi_j\in B_j$
such that $F^{n_j}(\xi_j)=\xi_j$, and it also holds
$d_{\mc{E}}(F^i(\xi_j),F^i(\zeta_0))\leq\chi$, for all $0\leq
i\leq n_j$. Hence, applying Lemma~\ref{Lemma-UniformContinuity}
and the fact that $\zeta_0, F^{n_j}(\zeta_0) \in X$, we prove that
$\xi_j$ is uniformly contracting along the fiber, as desired.

So, it only remains to prove Lemma~\ref{Lemma-PuntoCerca}:

\begin{proof}[Proof of Lemma~\ref{Lemma-PuntoCerca}]
  Since there exists a trivializing chart containing $x_0$ and $p_j$,
  there exists a point $\zeta_j\in\mc{E}_{p_j}$ such that
  $d_{\mc{E}}(\zeta_j,\zeta_0)=d(p_j,x)$.

  From the choice of $x_0$ and $p_j$, if $n_j$ is large enough, we can
  always assume that both $f^i(x_0)$ and $f^i(p_j)$ lie in the same
  trivializing chart for $0\leq i\leq n_j$. So, fixing a trivializing
  chart containing $f^{i}(x_0)$ and $f^{i}(p_j)$, we have a projection
  $\mathrm{pr}_2\colon\mc{E}_{f^i(p_j)}\to\mc{E}_{f^i(x_0)}$. Given
  two points $\xi\in\mc{E}_{f^i(x_0)}$ and $\eta\in\mc{E}_{f^i(p_j)}$
  such that $\xi,\mathrm{pr}_2(\eta)\in
  B_{F^i(\zeta_0)}'(\rho(F^i(\zeta_0)))$, we can define
  $d_{\mc{E}}'(\xi,\eta) := d(\pi(\xi), \pi(\eta)) + d'(\xi,
  \mathrm{pr}_2(\eta))$, were $d'$ is the distance in
  $B_{F^i(\zeta_0)}'(\rho(F^i(\zeta_0)))$ induced by the norm
  $\norm{\cdot}^\prime_{F^i(\zeta_0)}$.

  For $1\leq k \leq n_j$, and assuming that
  $d_{\mc{E}}'\big(F^{k-1}(\zeta_j),F^{k-1}(\zeta_0)\big)\leq \min \{
  \rho(F^{k-1}(\zeta_0)), \frac{\chi}{2} \}$, we can invoke
  Lemma~\ref{Lemma-UniformContinuity} to get
  \begin{equation}
    \label{eq:induction-argument}
    \begin{split}
      &d_{\mc{E}}'\big(F^k(\zeta_j),F^k(\zeta_0)\big)
      =d_{\mc{E}}'\Big(F\big(F^{k-1}(\zeta_j)\big), F\big(
      F^{k-1}(\zeta_0)\big)\Big) \leq \\
      & \leq \hat c e^{\eps \min\{k, n_j-k\}}
      d_{C^1}\big(F_{f^{k-1}(p_j)},F_{f^{k-1}(x_0)}\big) + \\
      & + e^\delta\norm{\dpf F(F^{k-1}(\zeta_0))}'
      d_{\mc{E}}'\big(F^{k-1}(\zeta_j), F^{k-1}(\zeta_0)\big),
    \end{split}
  \end{equation}
  where the constant $\hat c$ only depends on $A_X$ and $\rho_X$. The
  factor $\hat c e^{\eps\min\{k, n_j-k\}}$ appears to take into
  account the distortion in the new metric, which is bounded by $\hat
  c$ at the points $\zeta_0,F^{n_j}(\zeta_0) \in X$ and the change of
  the distortion at each iterate is bounded by $e^\eps$.

  Now, let us define the sequences $a_k:=\hat c e^{\eps \min\{k,
    n_j-k\}} d_{C^1}\big(F_{f^{k-1}(p_j)},F_{f^{k-1}(x_0)}\big)$ and
  $b_k:= e^\delta\norm{\dpf F(F^{k-1}(\zeta_0))}'$. Observe that
  $b_k\leq e^{(\lambda^+ + 3\eps)} < 1$, for every $k\geq 1$.

  By induction and applying estimate \eqref{eq:induction-argument},
  one gets
  \begin{equation}
    \label{eq:aj-bj}
    d_{\mc{E}}'\big( F^k(\zeta_j),F^k(\zeta_0)\big) \leq \sum_{i=1}^k
    a_i \Big(\prod_{j=i+1}^k 
    b_j\Big) \leq \sum_{i=1}^k a_i e^{(k-i)(\lambda^+ + 3\eps)}, 
    \quad\forall k\geq 1. 
  \end{equation}

  One can estimate the size of $a_i$ as follows (here is where
  H\"{o}lder continuity of $F$ is essential):
  \begin{equation}
    \label{eq:ai-estimate}
    a_i \leq e^{\eps \min \{ i,n_j-i \}} c' e^{-\alpha \lambda \min \{
      i, n_j-i \}} d(x_0,f^{n_j}(x_0))^\alpha, \quad\text{for } 0\leq
    i\leq n_j,
  \end{equation}
  where $c'>0$ depends on the constant $c$ appearing in Anosov closing
  lemma (Theorem \ref{thm:AnosovClosing}), the H\"{o}lder norm of the
  $C^{\alpha,1}$-bundle map $F$ and the constant $\hat c$ which was
  defined above.

  Choosing $n_j$ so that $d(x_0,f^{n_j}(x_0))$ is sufficiently small
  and recalling that $\eps < \frac{1}{2} \alpha \lambda$, thus we can
  perform the induction and to ensure that the iterates $F^k(\zeta_j)$
  of the point $\zeta_j$ remain always close enough to $F^k(\zeta_0)$.

  Using the estimate of Lemma \ref{Lemma-UniformContinuity}, one
  conclude that there is a ball $B_j$ with the desired properties.
\end{proof}

\section{Some questions}
\label{sec:questions}

In this section we pose some questions that emerge rather naturally
from our results and remain open. The firs one, taking into account
Theorem~\ref{thm:lyap-exp-dom-cobound}, should be the only remaining
step necessary to prove Liv\v{s}ic theorem for cocycles of
diffeomorphisms in any dimension:

\begin{question}
  Let $N\to\mc{E}\xrightarrow{\pi} M$ be a $C^{\alpha,1}$-fiber bundle
  and $F\colon\mc{E}\carr$ be $C^{\alpha,1}$-bundle map over a
  hyperbolic $C^\alpha$-homeomorphism $f\colon M\carr$. Suppose $F$
  satisfies \eqref{eq:POO-bund-map}. Then, is it true that $F$ is
  $\alpha$-dominated?
\end{question}

The second one concerns the validity of Liv\v{s}ic type results for
more general groups:

\begin{question}
  Does there exists a complete metric group $G$ such that Liv\v{s}ic
  theorem does not hold for $G$-cocyles? More precisely: do there
  exist a hyperbolic homeomorphism $f\colon M\carr$ and a
  $C^\alpha$-cocycle $\Phi\colon M\to G$ such that $\Phi$ satisfies
  \eqref{eq:POO} but it is not a $G$-coboundary?
\end{question}

\bibliographystyle{amsalpha} \bibliography{base-biblio}

\providecommand{\bysame}{\leavevmode\hbox to3em{\hrulefill}\thinspace}
\providecommand{\MR}{\relax\ifhmode\unskip\space\fi MR }
% \MRhref is called by the amsart/book/proc definition of \MR.
\providecommand{\MRhref}[2]{%
  \href{http://www.ams.org/mathscinet-getitem?mr=#1}{#2}
}
\providecommand{\href}[2]{#2}
\begin{thebibliography}{dlLW11}

\bibitem[ABC11]{AbdenurBonattiCrovisierNonunif}
F.~Abdenur, C.~Bonatti, and S.~Crovisier, \emph{Nonuniform hyperbolicity for
  {$C^1$}-generic diffeomorphisms}, Israel J. Math. \textbf{183} (2011), 1--60.
  \MR{2811152}

\bibitem[AV10]{AvilaVianaExtLyapInvPrin}
A.~Avila and M.~Viana, \emph{Extremal {L}yapunov exponents: an invariance
  principle and applications}, Invent. Math. \textbf{181} (2010), no.~1,
  115--189. \MR{2651382}

\bibitem[BCS13]{BonattiCrovisierShinohara}
C.~Bonatti, S.~Crovisier, and K.~Shinohara, \emph{The {$C^{1+\alpha}$}
  hypothesis in {Pesin} {Theory} revisited}, J. Mod. Dyn. \textbf{7} (2013),
  no.~4, 605--618. \MR{3177774}

\bibitem[dlLW10]{delaLlaveWindsorLivThmNonComm}
R.~de~la Llave and A.~Windsor, \emph{Liv\v sic theorems for non-commutative
  groups including diffeomorphism groups and results on the existence of
  conformal structures for {A}nosov systems}, Ergodic Theory Dynam. Systems
  \textbf{30} (2010), no.~4, 1055--1100. \MR{2669410}

\bibitem[dlLW11]{delaLlaveWindsorSmoorhDependence}
\bysame, \emph{Smooth dependence on parameters of solutions to cohomology
  equations over {A}nosov systems with applications to cohomology equations on
  diffeomorphism groups}, Discrete Contin. Dyn. Syst. \textbf{29} (2011),
  no.~3, 1141--1154. \MR{2773168}

\bibitem[GG14]{GrabarnikGuysinsky}
G.~Grabarnik and M.~Guysinsky, \emph{Liv\v{s}ic theorem for {Banach} rings},
  available at arXiv:1408.5639v1, 2014.

\bibitem[HPS77]{HirshPughShubInvMan}
M.~W. Hirsch, C.~C. Pugh, and M.~Shub, \emph{Invariant manifolds}, Lecture
  Notes in Mathematics, Vol. 583, Springer-Verlag, Berlin, 1977. \MR{0501173
  (58 \#18595)}

\bibitem[Jou88]{JourneRegLem}
J.-L. Journ{\'e}, \emph{A regularity lemma for functions of several variables},
  Rev. Mat. Iberoamericana \textbf{4} (1988), no.~2, 187--193. \MR{1028737
  (91j:58123)}

\bibitem[Kal11]{KalininLivThmMat}
B.~Kalinin, \emph{Liv\v sic theorem for matrix cocycles}, Ann. of Math.
  \textbf{173} (2011), no.~2, 1025--1042. \MR{2776369 (2012b:37082)}

\bibitem[Kat80]{KatokLyapExp}
A.~Katok, \emph{Lyapunov exponents, entropy and periodic orbits for
  diffeomorphisms}, Inst. Hautes \'Etudes Sci. Publ. Math. \textbf{51} (1980),
  137--173. \MR{573822 (81i:28022)}

\bibitem[KH96]{KatokHasselblatt}
A.~Katok and B.~Hasselblatt, \emph{Introduction to the modern theory of
  dynamical systems}, vol.~54, Cambridge Univ. Pr., 1996.

\bibitem[KK96]{KatokKononenkoCocStab}
A.~Katok and A.~Kononenko, \emph{Cocycle stability for partially hyperbolic
  systems}, Math. Res. Letters \textbf{3} (1996), 191--210.

\bibitem[KN11]{NiticaKatokRigHighRankI}
A.~Katok and V.~Ni{\c{t}}ic{\u{a}}, \emph{Rigidity in higher rank abelian group
  actions. {V}olume {I}}, Cambridge Tracts in Mathematics, vol. 185, Cambridge
  University Press, Cambridge, 2011, Introduction and cocycle problem.
  \MR{2798364 (2012i:37039)}

\bibitem[Koc13]{KocCohomC0Inst}
A.~Kocsard, \emph{On the cohomological {$C^0$-(in)stability}}, Bulletin of the
  Brazilian Mathematical Society \textbf{44} (2013), no.~3, 489--495.

\bibitem[Liv71]{LivCertPropHomol}
A.~N. Liv{\v{s}}ic, \emph{Certain properties of the homology of {$Y$}-systems},
  Mat. Zametki \textbf{10} (1971), 555--564. \MR{0293669}

\bibitem[Liv72]{LivsicCohomDynSys}
\bysame, \emph{Cohomology of dynamical systems}, Izv. Akad. Nauk SSSR Ser. Mat.
  \textbf{36} (1972), 1296--1320.

\bibitem[Ma{\~n}87]{ManeBook}
R.~Ma{\~n}{\'e}, \emph{Ergodic theory and differentiable dynamics}, Ergebnisse
  der Mathematik und ihrer Grenzgebiete (3), vol.~8, Springer-Verlag, Berlin,
  1987, Translated from the Portuguese by Silvio Levy. \MR{889254 (88c:58040)}

\bibitem[NP13]{NavasPonceLivAnalGerms}
A.~Navas and M.~Ponce, \emph{A {Livsic} type theorem for germs of analytic
  diffeomorphisms}, Nonlinearity \textbf{26} (2013), no.~1, 297--305.

\bibitem[NT95]{NiticaTorokCohomolDynSyst}
V.~Ni{\c{t}}ic{\u{a}} and A.~T{\"o}r{\"o}k, \emph{Cohomology of dynamical
  systems and rigidity of partially hyperbolic actions of higher-rank
  lattices}, Duke Math. J. \textbf{79} (1995), no.~3, 751--810. \MR{1355183
  (96k:58168)}

\bibitem[NT96]{NiticaTorokRegularResultsSolLiv}
\bysame, \emph{Regularity results for the solutions of the {L}ivsic cohomology
  equation with values in diffeomorphism groups}, Ergodic Theory Dynam. Systems
  \textbf{16} (1996), no.~2, 325--333. \MR{1389627}

\bibitem[Via08]{VianaAlmostAllCocyc}
M.~Viana, \emph{Almost all cocycles over any hyperbolic system have
  nonvanishing {L}yapunov exponents}, Ann. of Math. \textbf{167} (2008), no.~2,
  643--680. \MR{2415384 (2009i:37080)}

\bibitem[Wil13]{WilkinsonCohomoEq}
A.~Wilkinson, \emph{The cohomological equation for partially hyperbolic
  diffeomorphisms}, Ast\'erisque (2013), no.~358, 75--165. \MR{3203217}

\end{thebibliography}

\end{document}